\documentclass[12pt]{amsart}
\usepackage{amssymb,amsmath,color}
\usepackage{dsfont}
\usepackage{enumitem,amsrefs,hyperref}
\usepackage{color}
\usepackage{tikz}
\usepackage{pgf}
\usepackage{bm}
\usepackage{tkz-graph}
\usetikzlibrary{calc, positioning}
\usepackage{fullpage}
\renewcommand{\S}{\mathcal{S}}

\newcommand{\Sn}{\mathcal{S}_n}

\newtheorem{theorem}{Theorem}[section]
\newtheorem{conjecture}[theorem]{Conjecture}
\newtheorem{lemma}[theorem]{Lemma}

\theoremstyle{definition}
\newtheorem{definition}[theorem]{Definition}
\newtheorem{example}[theorem]{Example}

\theoremstyle{remark}

\date{\today}

\title{The Linear Relaxation of an Integer Program for the Union-Closed Conjecture}

\author{Brianna Amaral, Lucien Dalton, Drew Polakowski, \\ Annie Raymond, and Bertram Thomas}
\address{Department of Mathematics and Statistics,
Lederle Graduate Research Tower, 1623D,
University of Massachusetts Amherst
710 N. Pleasant Street
Amherst, MA 01003} \email{raymond@math.umass.edu}
\thanks{This paper is the result of an undergraduate research experience project at the University of Massachusetts, Amherst.} 
\begin{document}

\begin{abstract}
The Frankl conjecture, also known as the union-closed sets conjecture, states
that in any finite non-empty union-closed family, there exists an element in at least
half of the sets. Let $f(n,a)$ be the maximum number of sets in a union-closed family on a ground
set of $n$ elements where each element is in at most $a$ sets for some $a,n\in \mathbb{N}^+$. 
Proving that $f(n,a)\leq 2a$ for all $a, n \in \mathbb{N}^+$ is equivalent to proving the Frankl conjecture.
By considering the linear relaxation of the integer programming formulation that was proposed in \cite{PRT},
we prove that $O(a^2)$ is an upper bound for $f(n,a)$. We also provide different ways that this result could be strengthened. Additionally, we give a new proof that $f(n,2^{n-1}-1)=2^n-n$.
\end{abstract}

\maketitle

The union-closed sets conjecture was popularized by P\'eter Frankl in the late 1970's (\cite{Frankl83}), and is thus often referred to as the Frankl conjecture.

Throughout this paper, we think of $\mathcal{S}_n$ as being the power set $2^{[n]}$ where $[n]=\{1,2,\ldots, n\}$, and of $\mathcal{S}\subseteq \mathcal{S}_n$ as being a collection of distinct subsets of $[n]$ for some $n\in \mathbb{N}^+$. We say that $\mathcal{S}$ is union-closed if the union of any two sets in $\mathcal{S}$ is also a set in $\mathcal{S}$.

\begin{conjecture}[Frankl, 1979]\label{Frankl}
 If $\mathcal{S}\subseteq 2^{[n]}$ is union-closed and nonempty, then there exists an element in $[n]$ present in at least half of the sets of $\mathcal{S}$.
\end{conjecture}

\begin{example}
For example, $\mathcal{S}=\{\{1\}, \{2, 3\}, \{1, 2, 3\}, \{1, 2, 3, 4\}\}$ is union-closed, and elements $1$, $2$ and $3$ are all present in at least half of the four sets.
\end{example}

The Frankl conjecture has been well-studied from many different points of view. It is known to hold when $n\leq 11$ (\cite{BosMark08}) as well as when $|\mathcal{S}| \leq 46$ (\cite{RobSimp10}). A survey of results known for the conjecture was published in 2013 (\cite{BruhnSchaudt}).

In \cite{PRT}, the authors reformulated the Frankl conjecture as follows.

\begin{conjecture}\label{maximization}
Consider any $a, n\in \mathbb{N}^+$. Let $\mathcal{S}\subseteq 2^{[n]}$ be a union-closed family such that each element is in at most $a$ sets of $\mathcal{S}$. Then $|\mathcal{S}|\leq 2a$.
\end{conjecture}

This is indeed equivalent to Conjecture \ref{Frankl} since if there exists a union-closed family $\mathcal{S}$ where $|\mathcal{S}|>2a$ and every element is in at most $a$ sets of $\mathcal{S}$, then we would have found a counterexample to Conjecture \ref{Frankl}: every element would be in less than half of the sets of $\mathcal{S}$.

\begin{definition} Let $f(n,a)$ be the maximum number of sets in a union-closed family on a ground set of $n$ elements where each element is in at most $a$ sets for some $a,n\in \mathbb{N}^+$.
\end{definition}

Note that Conjecture \ref{maximization} can be reformulated as saying that $f(n,a)\leq 2a$ for all $a,n\in \mathbb{N}^+$.

In Section \ref{sec:IP}, we introduce some theory of linear and integer programming that is necessary to understand the rest of the paper. In Section \ref{sec:results}, we discuss the integer program model for $f(n,a)$ that was introduced in \cite{PRT} and we use a linear relaxation of that integer program to give an upper bound of $O(a^2)$ for $f(n,a)$. In Section \ref{sec:summer}, we show that $f(n,2^{n-1}-1)=2^n-n$.

Neither of these results are completely new. Indeed, in \cite{bbe}, the authors showed that the union-closed conjecture holds for any union-closed collection of sets where the number of sets is at least $\frac{2}{3}2^n$. Thus it was already known that $f(n,2^{n-1}-1)\leq  2^n-2$.

Moreover, in \cite{Knill}, it was shown that any union-closed collection on $m$ sets contains an element in at least $\frac{m-1}{\log_2 m}$ sets of the family. Our upper bound of $O(a^2)$ for $f(n,a)$ is weaker. However, assuming that Conjecture 13 of \cite{PRT} is true, we recover an equivalent bound. The interesting thing is the techniques used here could be improved in many simple ways, and thus potentially yield better results. These directions are discussed in \ref{sec:future}.

Thus, although the results in this paper are not new or impressive, the techniques used are completely different than those used in the papers above and still haven't been pushed to their full potential.


\section{An upper bound for $f(n,a)$}

The goal of this section is to provide an upper bound for $f(n,a)$ for all $a,n\in \mathbb{N}^+$. To do so, we consider an integer program that outputs $f(n,a)$ in \ref{sec:results}. We first introduce some necessary concepts from linear and integer programming in \ref{sec:IP}. 

\subsection{Linear and Integer Programming}\label{sec:IP}

Let $A\in \mathbb{R}^{m\times n}$, $\mathbf{b}\in \mathbb{R}^m$ and $\mathbf{c}\in \mathbb{R}^n$. The following is an \emph{integer program}: $\max\{\mathbf{c}^\top\mathbf{x}| A\mathbf{x} \leq \mathbf{b}, \mathbf{x}\in \mathbf{Z}^n\}$. Let $\lambda$ be the resulting \emph{optimal value} and say $\mathbf{x}^*$ is an \emph{optimal solution} if $A\mathbf{x}^* \leq \mathbf{b}$, $\mathbf{x}^* \in \mathbf{Z}^n$ and $\mathbf{c}^\top \mathbf{x}^* = \lambda$. Let $P:= \{\mathbf{x}\in \mathbf{Z}^n | A\mathbf{x}\leq \mathbf{b}\}$. Any $\mathbf{x}\in P$ is said to be a \emph{solution} of the integer program.

We now consider a \emph{linear relaxation} of the previous integer program: $\max\{\mathbf{c}^\top\mathbf{x}| A\mathbf{x} \leq \mathbf{b}, \mathbf{x}\in \mathbb{R}^n\}$. This program is a \emph{linear program}. Let $\bar{\lambda}$ be the resulting optimal value and say $\bar{\mathbf{x}}$ is an \emph{optimal solution} if $A\bar{\mathbf{x}} \leq \mathbf{b}$, $\bar{\mathbf{x}} \in \mathbb{R}^n$ and $\mathbf{c}^\top \bar{\mathbf{x}} = \bar{\lambda}$. Let $\bar{P}:= \{\mathbf{x}\in \mathbb{R}^n | A\mathbf{x}\leq \mathbf{b}\}$.  Any $\mathbf{x}\in \bar{P}$ is said to be a \emph{solution} of the linear program. 

First note that $P \subseteq \bar{P}$. Thus any optimal solution $\mathbf{x}^*$ for the integer program  is in $\bar{P}$ and $\mathbf{c}^\top x^* \leq \bar{\lambda}$. Thus $\lambda \leq \bar{\lambda}$, that is, the linear relaxation yields an upper bound to the original integer program. 

One advantage of the linear relaxation is that it can be solved in polytime through the \emph{ellipsoid method}, for example. No polytime algorithm to solve integer programs is known in general. Further, one can apply the theory of \emph{strong duality} to the linear relaxation. Through it, we obtain that, if $\bar{P}$ is not empty, then  $\bar{\lambda}=\min\{\mathbf{b}^\top \mathbf{y} | A^\top \mathbf{y} = \mathbf{b}, \mathbf{y} \geq \mathbf{0}, \mathbf{y}\in \mathbb{R}^m\}$, i.e., there exists a different linear program that yields the same optimal value.

Finally, consider one last linear program where we add additional constraints $$\min\{\mathbf{b}^\top \mathbf{y} | A^\top \mathbf{y} = \mathbf{b}, C\mathbf{y}=\mathbf{d}, \mathbf{y} \geq \mathbf{0}, \mathbf{y}\in \mathbb{R}^m\}$$ where $C\in \mathbb{R}^{k\times m}$ and $\mathbf{d}\in \mathbb{R}^m$. Let the optimal value of this linear program be $\tilde{\lambda}$. Note that any optimal solution $\tilde{\mathbf{y}}$ of this linear program is a solution of the previous linear program. Therefore, $\tilde{\lambda}\geq\bar{\lambda}\geq \lambda$ and this last linear program also yields an upper bound to the original integer program.

\subsection{Results}\label{sec:results}
In \cite{PRT}, the authors introduced the following integer program that computes $f(n,a)$ for any fixed $n, a\in \mathbb{N}^+$.

\begin{align*}
f(n,a) = \max & \sum_{S\in \mathcal{S}_n} x_S & \\
\textup{s.t. } & x_S+x_T \leq 1+x_{S\cup T} & \forall S, T \in \mathcal{S}_n \\
& \sum_{\substack{S\in \mathcal{S}_n: \\ e\in S}} x_S \leq a & \forall e\in [n] \\
& x_S \in \{0,1\} & \forall S\in \mathcal{S}_n
\end{align*}

The claim is that $\mathcal{S}:=\{S \in 2^{[n]}|x_S=1\}$ is a union-closed family. Indeed, from the first set of constraints, if sets $S$ and $T$ are present in $\mathcal{S}$, then the associated variables will be one, and thus $x_{S\cup T}$ must also be one, meaning that $S\cup T$ must also be in $\mathcal{S}$. If either $S$ or $T$ is not present, then there is no restriction on whether $S\cup T$ must be in the collection. The second set of constraints ensures that each element is in at most $a$ sets of the collection. Finally, the total number of sets in such a union-closed collection is maximized by the objective function.

The following lemma from \cite{PRT} (Proposition 12.1 and Theorem 20 in that paper) is useful in restricting which $f(n,a)$'s need to be studied.

\begin{lemma}{\cite{PRT}}\label{diagonal}
In general, $f(n,a)\leq f(n+1,a)$ for all $a, n\in \mathbb{N}$. Moreover, $f(n,a)=f(n+1,a)$ for all $n\geq a-1$.
\end{lemma}

Note that this implies that for a fixed $a\in \mathbb{N}$, $f(n,a)\leq f(a,a)$ for all $n\in \mathbb{N}$. Integer programming solvers such as Gurobi or Cplex can compute the values of $f(a,a)$ up to 8.

\[
\begin{array}{|c|c|}
\hline
a & f(a,a)\\
\hline
1 & 2\\
2 & 4\\
3& 5\\
4 & 8\\
5 & 9\\
6 & 10\\
7 & 12\\
8 & 16\\
\hline
\end{array}
\]

For other values, we give the following upper bound by applying the concepts of linear and integer programming discussed in the previous subsection. We will consider the dual of $f(n,a)$ and add constraints requiring all variables corresponding to union-closed inequalities involving sets of some fixed cardinalities to be the same. 

\begin{theorem}\label{thm:upperbound}
We have that $f(n,a)\leq \frac{5a^4-12a^3+31a^2-24a+48}{12(a^2-3a+4)}$ for all $a\in \mathbb{Z}$ and $n\geq 7$.
\end{theorem}

\begin{proof}
Let
\begin{align*}
\alpha & = 1-\frac{2\binom{n-1}{2}}{3+3\binom{n-1}{2}} = \frac{n^2-3n+8}{3n^2-9n+12}\\
\beta & = \frac{2}{3+3\binom{n-1}{2}}=\frac{4}{3n^2-9n+12}\\
\gamma &=\frac{1}{\binom{n-2}{2}}\left(-1 + \frac{2(n-2)^2}{3+3\binom{n-1}{2}}\right).
\end{align*}

Note that $\gamma\geq 0$ if $n\geq 7$ and $\alpha, \beta \geq 0$ for all $n\geq 0$. 

We claim that the linear combination obtained by taking
\begin{align*}
&\sum_{e\in [n]} \alpha \left(\sum_{\substack{S\in \mathcal{S}_n:\\e\in S}} x_S \leq a\right)\\
+ &\sum_{\substack{S,T\in \mathcal{S}_n:\\|S|=1, |T|=2\\ |S\cup T|= 3}} \beta \left(x_S+x_T -x_{S\cup T} \leq 1\right) \\
+ &\sum_{\substack{S, T\in \mathcal{S}_n:\\|S|=2, |T|=2\\ |S\cup T|= 4}} \gamma \left(x_S+x_T -x_{S\cup T}  \leq 1 \right) \\
+ & x_\emptyset \leq 1
\end{align*}
yields
$$\sum_{S\in \mathcal{S}} c_S x_S \leq \bar{f}(n,a)$$
where each $c_S \geq 1$ and $\bar{f}(n,a)=n\cdot a \cdot \alpha + 3\binom{n}{3} \cdot 1 \cdot \beta + 3\binom{n}{4}\cdot 1\cdot \gamma + 1$. 

Let's check this by calculating the coefficient for sets $S$ of different size. Let's call inequalities $\sum_{\substack{S\in \mathcal{S}:\\e\in S}} x_S \leq a$ \emph{frequency inequalities}, $x_S+x_T -x_{S\cup T}\leq 1$ where $|S|=1$, $|T|=2$, $|S\cup T|= 3$ \emph{123-union-closed inequalities}, and $x_S+x_T -x_{S\cup T}\leq 1$ where $|S|=2$, $|T|=2$, $|S\cup T|= 4$ \emph{224-union-closed inequalities}.  
\begin{itemize}
\item $|S|=0$: the empty set only appears once with a coefficient of 1.
\item $|S|=1$: any $1$-element set will appear in exactly one frequency inequality and $\binom{n-1}{2}$ 123-union-closed inequalities, and no 224-union-closed inequalities. Thus the coefficient for any $1$-element will be $1\alpha + \binom{n-1}{2} \beta = 1$. 
\item $|S|=2$: any $2$-element set will appear in exactly two frequency inequalities and $\binom{n-2}{1}$ 123-union-closed inequalities and $\binom{n-2}{2}$ 224-union-closed inequalities. Note that it always appear positively. Thus the coefficient for any $2$-element set will be $2\alpha + \binom{n-2}{1} \cdot \beta + \binom{n-2}{2}\gamma = 1$.
\item $|S|=3$: any $3$-element set will appear in exactly three frequency inequalities. It will also appear negatively in three union-closed 123-union-closed inequalities, and zero 224-union-closed inequality. Thus any $3$-element set will have coefficient $3\alpha-3\beta=1$.
\item $|S|=4$: any $4$-element set will appear in exactly four frequency inequalities. It will also appear negatively in three 224-union-closed inequalities, and zero 123-union-closed inequality. Thus any $4$-element set will have coefficient $4\alpha - 3\gamma\geq 1$.
\item $|S|=i, i\geq 5$: any $i$-element set will appear in exactly five frequency inequalities and nowhere else. Then $c_S=\frac{5n^2-15n+40}{3n^2-9n+12}$ which is always at least 1.
\end{itemize}

Finally, note that in our linear combination, we take $n$ frequency inequalities, $3\binom{n}{3}$ 123-union-closed inequalities, $3\binom{n}{4}$ 224-union-closed inequalities and one empty set inequality. Thus

$$\bar{f}(n,a)=n\cdot a \cdot \alpha + 3\binom{n}{3} \cdot 1 \cdot \beta + 3\binom{n}{4}\cdot 1\cdot \gamma + 1.$$ 

Since $x_S\geq 0$ for all $S\in \mathcal{S}$, $\sum_{S\in \mathcal{S}} x_S \leq \sum_{S\in \mathcal{S}} c_S x_S$, and so $\bar{f}(n,a)$ is an upper bound for $f(n,a)$.

By Lemma~\ref{diagonal}, we know that, for a fixed $a\in \mathbb{N}$, $f(n,a)\leq f(a,a)$ for all $n\in \mathbb{N}$. Thus, finding an upper bound for $f(a,a)$ yields an upper bound for all $f(n,a)$. Thus $$\bar{f}(a,a)=\frac{5a^4-12a^3+31a^2-24a+48}{12(a^2-3a+4)}$$ is an upper bound for $f(n,a)$ for all $n\in \mathbb{N}$. 

\end{proof}

To give the reader a better grasp on this upper bound, here is a table compiling a few values of $\lfloor \bar{f}(a,a) \rfloor$.

\[
\begin{array}{|c|c|}
\hline
a & \lfloor \bar{f}(a,a) \rfloor \\
\hline
7 & 24\\
8 & 30\\
9 & 37\\
10 & 46\\
11 & 55\\
12 & 64\\
13 & 75\\
14 & 86 \\
15 & 99\\
16 & 112\\
\hline
\end{array}
\]

\subsection{Future directions}\label{sec:future}
We first note that the result we found in Theorem \ref{thm:upperbound} is an upper bound for the linear relaxation of $f(n,a)$ where we replace $x_e \in \{0,1\}$ by $0\leq x_e \leq 1$. In other words, we are giving an upper bound to an upper bound of $f(n,a)$, namely to its linear relaxation $f_r(n,a)$. For example, $\lfloor f_r(8,8)\rfloor =20<\lfloor \bar{f}(8,8) \rfloor = 30$ and $\lfloor f_r(9,9)\rfloor =26 < \lfloor \bar{f}(9,9)\rfloor =36$. To find this upper bound for the linear relaxation, we considered its dual and added constraints that required that all variables corresponding to union-closed inequalities involving sets of some fixed cardinalities $a$, $b$ and $c$ be the same. This is very restrictive. Thus it might be possible to give a better upper bound for the linear relaxation of $f(n,a)$ or even to find its exact value. 

Furthermore, the linear relaxation itself gets weaker as $n$ increases. By adding valid linear inequalities, one can strengthen the linear relaxation. A few are discussed in \cite{thesis}. 

Finally, note that the formula we found for $\bar{f}(n,a)$ in the proof of \ref{thm:upperbound} is for any $n,a$ with $n\geq 7$, and not only for $n=a$. In \cite{PRT}, the authors conjectured that $f(n,a)=f(n+1,a)$ for all $n\geq \lceil \log_2 a \rceil +1$, i.e., for all values of $n$ for which it makes sense to compute $f(n,a)$ given some particular $a$. If that conjecture is true, then $f(n,a)$ is upper bounded by $f(\lceil \log_2 a \rceil + 1,a)$ and thus by $\bar{f}(\lceil \log_2 a \rceil+1, a)$ for all $a$. Note that $\bar{f}(\lceil \log_2 a \rceil+1, a)$ yields an upper bound similar to Knill's lower bound that states that for any union-closed family with $m$ sets, there exists an element in at least $\frac{m-1}{\log_2 m}$ sets.

We believe that these techniques have much more to offer. Despite all the simplifications we used, they still led to some results. By removing the harshest of these simplifications, one might obtain new and interesting results for the Frankl conjecture.

\section{A proof of $f(n,2^{n-1}-1)=2^n-n$}\label{sec:summer}

\begin{definition}
Fix $n$ and $m$. Then let $g(n,m)$ be the minimum number of sets containing the most frequent element in a union-closed family of $m$ sets on $n$ elements. 
\end{definition}

\begin{lemma}\label{lem:missingsubsets}
Consider a union-closed family that does not contain some set $S$ where $|S|\geq 2$. Then the family contains at most one set $T\subset S$ such that $|T|=|S|-1$.
\end{lemma}

\proof
Suppose not: suppose there exist sets $T_1$ and $T_2$ in the family such that $T_1,T_2\subset S$ and $|T_1|=|T_2|=|S|-1$. Then $T_1\cup T_2=S$, and so $S$ would have to be in the family as well since it is union-closed, a contradiction. 
\qed

\smallskip

\begin{lemma}\label{lem:missingcovering}
Let $\mathcal{S}$ be a union-closed family on $n$ elements, and let $\Sn\backslash \S=\{S_1, \ldots, S_k\}$. If $S_1\cup \ldots\cup S_k \supseteq \{e_1, \ldots, e_l\}\neq \emptyset$ for some $e_1, e_2, \ldots, e_l\in [n]$, then $k\geq l$.
 \end{lemma}

\begin{proof}
We show this by induction on $l$. If $l=1$, then $\Sn\backslash \S$ cannot be empty, so $k\geq 1$. (Similarly, if $l=2$, then one cannot simply remove one set containing both $e_1$ and $e_2$ since then $\S$ would not be union-closed. Thus, $k\geq 2$.)

Now suppose this holds up to $l-1$, and we will show it for $l$.  Among $S_1, S_2, \ldots, S_k$, let $S^*$ be a set of maximum cardinality.

\textbf{Case 1:} Suppose  $2 \leq |S^*| \leq l-1$. Let $S_{i_1}, \ldots, S_{i_t} \in \Sn \backslash \S$ be such that $S_{i_j}$ contains no other nonempty set in $\Sn\backslash\S$ and $S_{i_j}\not\subseteq S^*$. Note that $S^*\cup S_{i_1}\cup \ldots \cup S_{i_t} \supseteq \{e_1, e_2, \ldots, e_l\}$ since any element $e_j$ is in at least one set of $\Sn\backslash\S$, and certainly a set $\bar{S}$ in $\Sn\backslash\S$ of minimum cardinality containing $e_j$ contains no other nonempty set in $\Sn\backslash\S$. Either we picked $\bar{S}$ or it is a subset of $S^*$; in both cases, $e_j$ will be in the union.  

By Lemma~\ref{lem:missingsubsets}, since $|S^*|\geq 2$, there are at least $|S^*|-1$ subsets of $S^*$ of size $|S^*|-1$ that are also in $\Sn \backslash \S$. Note that none of these subsets is a set that we kept since we did not keep any set that is a subset of $S^*$. 

Note that the collection of sets $S_{i_1}, \ldots, S_{i_t}$ is such that $$S_{i_1}\cup \ldots \cup S_{i_t} \supseteq \{e_1, e_2, \ldots, e_l\}\backslash S^*.$$ Furthermore, note that there is a union-closed family $\S'$ on $n$ elements such that $\Sn\backslash\S'=\{S_{i_1}, \ldots, S_{i_t}\}$. Indeed, it cannot be that there exists $T_1$ and $T_2$ in $\S'$ such that $T_1\cup T_2=S_{i_j}$ for some $1\leq j \leq t$. If either $T_1$ or $T_2$ had been in $\Sn\backslash \S$, then we would not have kept $S_{i_j}$ since $T_1$ and $T_2$ are subsets of that set. Then that means that $T_1$ and $T_2$ were both in $\S$, but then $\S$ would not have been union-closed since their union $S_{i_j}$ was not in $\S$. 

Thus, we can use the induction hypothesis to deduce that $t\geq l-|S^*|$. So we have found that there are at least these $l-|S^*|$ sets in $\Sn\backslash \S$, as well as $S^*$ itself and its $|S^*|-1$ subsets, for a total of $l$ sets as desired.

\textbf{Case 2:} Suppose $|S^*|=1$. Then all sets $S_1, \ldots, S_k$ are singletons (or the empty set), so to cover $l$ elements, one needs at least $l$ sets, so $k\geq l$ as desired. 

\textbf{Case 3:} Suppose $|S^*|=l$. By Lemma~\ref{lem:missingsubsets}, at least $l-1$ subsets of $S^*$ of cardinality $l-1$ are also in $\Sn\backslash\S$, meaning that there are also at least $l$ sets in $\Sn \backslash \S$ as desired.
\end{proof}

\begin{theorem}\label{thm:g}
The following holds: $g(n,2^n-i)=2^{n-1}$ for $0\leq i \leq n-1$.
\end{theorem}

\begin{proof}
We will show that for any union-closed family $\S$ of $m$ sets on $n$ elements where $m=2^n-i$ for some $0 \leq i \leq n-1$, there exists an element in $2^{n-1}$ sets. In other words, we will show that there is an element that is not in any of the sets in $\Sn\backslash \S$. 

By Lemma~\ref{lem:missingcovering}, if the sets in $\Sn\backslash\S$ covered $[n]$, there would have to be at least $n$ sets in $\Sn\backslash \S$. But we know that $|\Sn\backslash\S|=i$ for  some $0\leq i \leq n-1$. Thus the sets in $\Sn\backslash\S$ cannot cover $[n]$ and there is an element that is not in any of the sets in $\Sn\backslash\S$. 
\end{proof}

Similarly, one can show that $f(n,2^{n-1}-1)=2^n-n$. Note that it is clear that $f(n,2^{n-1})=2^n$ as one can take the whole power set of $n$.

\begin{theorem}
We have that $f(n,2^{n-1}-1)=2^n-n$ for all $n\in \mathbb{N}^+$.
\end{theorem}

\begin{proof}
Suppose that $f(n,2^{n-1}-1)=2^n-n+k$ for some $k\in [n]$. Then, by Proposition 12.5 of \cite{PRT}, $g(n,2^n-n+k)=2^{n-1}-1$ which is a contradiction to \ref{thm:g}. Thus, we have that $f(n,2^{n-1}-1)\leq 2^n-n$.

Now consider the power set $\mathcal{S}_n$. Each element is in exactly $2^{n-1}$ sets. Remove the $n$ singletons, i.e., the $n$ sets containing exactly one element. The family one thus obtains is still union-closed, has $2^n-n$ sets, and each element is in $2^{n-1}-1$ sets. Therefore, we also have that $f(n,2^{n-1}-1)\geq 2^n-n$, and so the theorem holds.
\end{proof}

\bibliographystyle{alpha}
\bibliography{reufrankl}

\end{document}